\documentclass[reqno]{amsart}
\usepackage{url}
\usepackage[dvips]{graphicx}

\newtheorem{theorem}{Theorem}
\newtheorem{lemma}{Lemma}
\newtheorem{definition}{Definition}
\begin{document}
\bibliographystyle{plain}

\title[Constructive proof of the existence of equilibrium]{Constructive proof of the existence of equilibrium in a competitive economy with sequentially locally non-constant excess demand functions}

\author{Yasuhito Tanaka}
\address{Faculty of Economics, Doshisha University, Kamigyo-ku, Kyoto, 602-8580, Japan}
\email{yasuhito@mail.doshisha.ac.jp}
\thanks{This research was partially supported by the Ministry of Education, Science, Sports and Culture of Japan, Grant-in-Aid for Scientific Research (C), 20530165.}\address{Faculty of Economics, Doshisha University, Kamigyo-ku, Kyoto, 602-8580, Japan}

\date{}



\maketitle

\begin{abstract}
We present a constructive proof of the existence of an equilibrium in a competitive economy with sequentially locally non-constant excess demand functions. And we will show that the existence of such an equilibrium implies Sperner's lemma. Since the existence of an equilibrium is derived from the existence an approximate fixed point of uniformly continuous functions, which is derived from Sperner's lemma, the existence of an equilibrium in a competitive economy with sequentially locally non-constant excess demand functions is equivalent to Sperner's lemma.
\end{abstract}

\section{Introduction}
\label{sec:intro}

It is well known that Brouwer's fixed point theorem can not be constructively proved\footnote{\cite{kel} provided a \emph{constructive} proof of Brouwer's fixed point theorem. But it is not constructive from the view point of constructive mathematics \'{a} la Bishop. It is sufficient to say that one dimensional case of Brouwer's fixed point theorem, that is, the intermediate value theorem is non-constructive. See \cite{br} or \cite{da}.}. Thus, the existence of a competitive equilibrium also can not be constructively proved. Sperner's lemma which is used to prove Brouwer's theorem, however, can be constructively proved. Some authors have presented an approximate version of Brouwer's theorem using Sperner's lemma. See \cite{da} and \cite{veld}.  Thus, Brouwer's fixed point theorem is constructively, in the sense of constructive mathematics \'{a} la Bishop, proved in its approximate version.

Also \cite{da} states a conjecture that a uniformly continuous function $\varphi$ from a simplex into itself, with property that each open set contains a point $x$ such that $x\neq f(x)$, which means $|x-f(x)|>0$, and also at every point $x$ on the boundaries of the simplex $x\neq f(x)$, has an exact fixed point. We call such a property of functions \emph{local non-constancy}. Recently \cite{berger} showed that the following theorem is equivalent to Brouwer's fan theorem.
\begin{quote}
Each uniformly continuous function $\varphi$ from a compact metric space $X$ into itself with at most one fixed point and approximate fixed points has a fixed point.
\end{quote}
By reference to the notion of \emph{sequentially at most one maximum} in \cite{berg} we consider a condition that a function $\varphi$ is \emph{sequentially locally non-constant}. Sequential local non-constancy is stronger than the condition in \cite{da} (local non-constancy), and is different from the condition that a function has \emph{at most one fixed point} in \cite{berg}.

\cite{orevkov} constructed a computably coded continuous function $\varphi$ from the unit square into itself, which is defined at each computable point of the square, such that $\varphi$ has no computable fixed point. His map consists of a retract of the computable elements of the square to its boundary followed by a rotation of the boundary of the square. As pointed out by \cite{hirst}, since there is no retract of the square to its boundary, his map does not have a total extension.

In this paper we present a proof of the existence of an exact equilibrium in a competitive  economy with \emph{sequentially locally non-constant} excess demand functions. Also we will show that the existence of such an equilibrium implies Sperner's lemma. Since the existence of an equilibrium is derived from the existence an approximate fixed point of uniformly continuous functions, which is derived from Sperner's lemma, the existence of an equilibrium in a competitive economy with sequentially locally non-constant excess demand functions is equivalent to Sperner's lemma.

In the next section we present our theorem and its proof. In Section 3 we will show that the existence of an equilibrium in a competitive economy with sequentially locally non-constant excess demand functions implies Sperner's lemma.

\section{Existence of equilibrium in a competitive economy}

We consider an $n$-dimensional simplex $\Delta$ as a compact metric space. Let $\mathbf{p}$ be a point in $\Delta$, and consider a uniformly continuous function $\varphi$ from $\Delta$ into itself.

 According to \cite{da} and \cite{veld} $\varphi$ has an approximate fixed point. It means 
\[\mathrm{For\ each}\ \varepsilon>0\ \mathrm{there\ exists}\ \mathbf{p}\in \Delta\ \mathrm{such\ that}\ |\mathbf{p}-\varphi(\mathbf{p})|<\varepsilon.\]
Since $\varepsilon>0$ is arbitrary,
\[\inf_{\mathbf{p}\in \Delta}|\mathbf{p}-\varphi(\mathbf{p})|=0.\]

The definition of local non-constancy of functions in a simplex $\Delta$ is as follows;
\begin{definition}(Local non-constancy of function)
\begin{enumerate}
	\item At a point $\mathbf{p}$ on the faces (boundaries) of $\Delta$ $\varphi(\mathbf{p})\neq \mathbf{p}$. This means $\varphi_i(\mathbf{p})>\mathbf{p}_i$ or $\varphi_i(\mathbf{p})<\mathbf{p}_i$ for at least one $i$, where $\mathbf{p}_i$ and $\varphi_i(\mathbf{p})$ are the $i$-th components of $\mathbf{p}$ and $\varphi(\mathbf{p})$.
	\item And in any open set in $\Delta$ there exists a point $\mathbf{p}$ such that $\varphi(\mathbf{p})\neq \mathbf{p}$.
\end{enumerate}
\end{definition}

On the other hand, the notion that $\varphi$ has at most one fixed point in \cite{berger} is defined as follows;
\begin{definition}[At most one fixed point]
For all $\mathbf{p}, \mathbf{q}\in \Delta$, if $\mathbf{p}\neq \mathbf{q}$, then $\varphi(\mathbf{p})\neq \mathbf{p}$ or $\varphi(\mathbf{q})\neq \mathbf{q}$.
\end{definition}

By reference to the notion of \emph{sequentially at most one maximum} in \cite{berg},  we define the property of \emph{sequential local non-constancy}.

First we recapitulate the compactness (total boundedness with completeness) of a set in constructive mathematics. $\Delta$ is totally bounded in the sense that for each $\varepsilon>0$ there exists a finitely enumerable $\varepsilon$-approximation to $\Delta$\footnote{A set $S$ is finitely enumerable if there exist a natural number $N$ and a mapping of the set $\{1, 2, \dots, N\}$ onto $S$.}. An $\varepsilon$-approximation to $\Delta$ is a subset of $\Delta$ such that for each $\mathbf{p}\in \Delta$ there exists $\mathbf{q}$ in that $\varepsilon$-approximation with $|\mathbf{p}-\mathbf{q}|<\varepsilon$. Each face (boundary) of $\Delta$ is also a simplex, and so it is totally bounded. According to Corollary 2.2.12 of \cite{bv} we have the following result.
\begin{lemma}
For each $\varepsilon>0$ there exist totally bounded sets $H_1, H_2, \dots, H_n$, each of diameter less than or equal to $\varepsilon$, such that $\Delta=\cup_{i=1}^nH_i$.\label{closed}
\end{lemma}
Since $\inf_{\mathbf{p}\in \Delta}|\mathbf{p}-\varphi(\mathbf{p})|=0$, we have $\inf_{\mathbf{p}\in H_i}|\mathbf{p}-\varphi(\mathbf{p})|=0$ for some $H_i\subset \Delta$.

The definition of sequential local non-constancy is as follows;
\begin{definition}(Sequential local non-constancy of functions)
There exists $\bar{\varepsilon}>0$ with the following property. For each $\varepsilon>0$ less than or equal to $\bar{\varepsilon}$ there exist totally bounded sets $H_1, H_2, \dots, H_m$, each of diameter less than or equal to $\varepsilon$, such that $\Delta=\cup_{i=1}^m H_i$, and if for all sequences $(\mathbf{p}_n)_{n\geq 1}$, $(\mathbf{q}_n)_{n\geq 1}$ in each $H_i$, $|\varphi(\mathbf{p}_n)-\mathbf{p}_n|\longrightarrow 0$ and $|\varphi(\mathbf{q}_n)-\mathbf{q}_n|\longrightarrow 0$, then $|\mathbf{p}_n-\mathbf{q}_n|\longrightarrow 0$.
\end{definition}

Now we show the following lemma, which is based on Lemma 2 of \cite{berg}.
\begin{lemma}
Let $\varphi$ be a uniformly continuous function from $\Delta$ into itself. Assume $\inf_{\mathbf{p}\in H_i}\varphi(\mathbf{p})=0$ for some $H_i\subset \Delta$ defined above. If the following property holds:
\begin{quote}
For each $\delta>0$ there exists $\eta>0$ such that if $\mathbf{p}, \mathbf{q}\in H_i$, $|\varphi(\mathbf{p})-\mathbf{p}|<\eta$ and $|\varphi(\mathbf{q})-\mathbf{q}|<\eta$, then $|\mathbf{p}-\mathbf{q}|\leq \delta$.
\end{quote}
Then, there exists a point $\mathbf{r}\in H_i$ such that $\varphi(\mathbf{r})=\mathbf{r}$. \label{fix0}
\end{lemma}
\begin{proof}
Choose a sequence $(\mathbf{p}_n)_{n\geq 1}$ in $H_i$ such that $|\varphi(\mathbf{p}_n)-\mathbf{p}_n|\longrightarrow 0$. Compute $N$ such that $|\varphi(\mathbf{p}_n)-\mathbf{p}_n|<\eta$ for all $n\geq N$. Then, for $m, n\geq N$ we have $|\mathbf{p}_m-\mathbf{p}_n|\leq \delta$. Since $\delta>0$ is arbitrary, $(\mathbf{p}_n)_{n\geq 1}$ is a Cauchy sequence in $H_i$, and converges to a limit $\mathbf{r}\in H_i$. The continuity of $\varphi$ yields $|\varphi(\mathbf{r})-\mathbf{r}|=0$, that is, $\varphi(\mathbf{r})=\mathbf{r}$.
\end{proof}

Consider a competitive exchange economy. There are $n+1$ goods $X_0$, $X_1$, $\cdots$, $X_n$. $n$ is a finite positive integer. The prices of the goods are denoted by $p_i(\geq 0),\ i=0,1,\cdots,n$. Let $\bar{p}=p_0+p_1+\cdots +p_n$, and define
\[ \bar{p}_i=\frac{p_i}{\bar{p}},\ i=0,1,\cdots,n.\]
Denote afresh $\bar{p}_0$, $\bar{p}_1$, $\cdots$, $\bar{p}_n$, respectively, by $p_0$, $p_1$, $\cdots$, $p_n$. Then,
\begin{equation}
p_0+p_1+\cdots +p_n=1. \label{wal1}
\end{equation}
$\mathbf{p}=(p_0,p_1,\cdots,p_n)$ represents a point on an $n$-dimensional simplex. It is usually assumed that consumers' excess demand (demand minus supply) for each good is homogeneous of degree zero. It means that consumers' excess demand for each good is determined by relative prices of the goods, and above notation of the prices yields no loss of generality. We denote the vector of excess demands for the goods when the vector of prices is $\mathbf{p}$ by $\mathbf{f}(\mathbf{p})=(f_1, f_2, \dots, f_n)$. We require the following condition;
\begin{equation}
\mathbf{p}\mathbf{f}(\mathbf{p})=p_0f_0+p_1f_1+\cdots +p_nf_n=0\ \mathrm{(Walras\ Law)}. \label{wal}
\end{equation}
$f_i$ is equal to the sum of excess demands of all consumers for the good $X_i$. By the budget constraint for each consumer, in a competitive exchange economy the sum of excess demands of all consumers for each good must be 0. Adding the budget constraints of all consumers yields (\ref{wal}). We assume that the excess demand function $\mathbf{f}(\mathbf{p})$ is uniformly continuous about the prices of the goods. Uniform continuity of $\mathbf{f}$ means that for any $\mathbf{p}$, $\mathbf{p}'$ and $\eta>0$ there exists $\delta>0$ such that
\[\mathrm{If}\ |\mathbf{p}'-\mathbf{p}|<\delta, \mathrm{we\ have}\ |\mathbf{f}(\mathbf{p}')-\mathbf{f}(\mathbf{p})|<\eta.\]
$\delta$ depends only on $\eta$ not on $\mathbf{p}$ and $\mathbf{p}'$. It implies that a slight price change yields only a slight excess demand change. An equilibrium of a competitive exchange economy is a state where excess demand for each good is not positive.

Next we assume the following condition.
\begin{definition}[Sequential local non-constancy of excess demand functions]
There exists $\bar{\varepsilon}>0$ with the following property. For each $\varepsilon>0$ less than or equal to $\bar{\varepsilon}$ there exist totally bounded sets $H_1, H_2, \dots, H_m$, each of diameter less than or equal to $\varepsilon$, such that $\Delta=\cup_{i=1}^m H_i$, and if for all sequences $(\mathbf{p}_n)_{n\geq 1}$, $(\mathbf{q}_n)_{n\geq 1}$ in each $H_i$, $\max(f_i(\mathbf{p}_n),0)\longrightarrow 0$ and $\max(f_i(\mathbf{q}_n),0)\longrightarrow 0$ for all $i$, then $|\mathbf{p}_n-\mathbf{q}_n|\longrightarrow 0$. \label{sld}
\end{definition}
This condition implies isolatedness of equilibria.

Consider the following function from the set of price vectors $\mathbf{p}=(p_0,p_1,\cdots,p_n)$ to the set of $n+1$ tuples of real numbers $\mathbf{v}=(v_0,v_1,\cdots,v_n)$.
\[v_i=p_i+\max(f_i,0)\ \mathrm{for\ all}\ i.\]
With this we define a function from an $n$-dimensional simplex $\Delta$ to itself $\varphi(\mathbf{p})=(\varphi_0,\varphi_1,\cdots,\varphi_n)$ as follows;
\[ \varphi_i=\frac{1}{v_0+v_1+\cdots +v_n}v_i,\ \mathrm{for\ all}\ i\]
Since $\varphi_i\geq 0,\ i=0, 1, \cdots, n$ and
\[\varphi_0+\varphi_1+\cdots + \varphi_n=1,\]
$(\varphi_0,\varphi_1,\cdots,\varphi_n)$ represents a point on $\Delta$. By the uniform continuity of $\mathbf{f}$, $\varphi$ is also uniformly continuous. 

If $|\varphi(\mathbf{p}_n)-\mathbf{p}_n|\longrightarrow 0$, then $\left|\frac{p_i+\max(f_i, 0)}{1+\sum_{j=0}^n\max(f_j, 0)}\right|_n\longrightarrow p_i$. From $\sum_{j=0}^np_j=1$ there is a $k$ such that $p_k>0$. If for all such $k$ $\max(f_k,0)=f_k>0$ holds, we can not cancel out $p_kf_k>0$ because the price of any good can not be negative, and the Walras law (\ref{wal}) is violated. Therefore, $\max(f_i, 0)\longrightarrow 0$ for all $i$.

Thus, from the sequential local non-constancy of excess demand functions we obtain the following results;
\begin{quote}
For $\bar{\varepsilon}>0$ defined in Definition \ref{sld} and each $\varepsilon>0$ less than or equal to $\bar{\varepsilon}$ there exist totally bounded sets $H_1, H_2, \dots, H_m$, each of diameter less than or equal to $\varepsilon$, such that $\Delta=\cup_{i=1}^m H_i$, and if for all sequences $(\mathbf{p}_n)_{n\geq 1}$, $(\mathbf{q}_n)_{n\geq 1}$ in each $H_i$, $|\varphi(\mathbf{p}_n)-\mathbf{p}_n|\longrightarrow 0$ and $|\varphi(\mathbf{q}_n)-\mathbf{q}_n|\longrightarrow 0$, then $|\mathbf{p}_n-\mathbf{q}_n|\longrightarrow 0$.
\end{quote}
Therefore, $\varphi$ is a sequentially locally non-constant function. 

Now we show the following theorem.
\begin{theorem}
In a competitive exchange economy, if the excess demand functions are uniformly continuous about the prices of the goods, sequentially locally non-constant and satisfy the Walras law, then there exists an equilibrium.
\end{theorem}
\begin{proof}
\begin{enumerate}
\item Choose a sequence $(\mathbf{r}_n)_{n\geq 1}$ in $H_i$ defined above such that $|\varphi(\mathbf{r}_n)-\mathbf{r}_n|\longrightarrow 0$. In view of Lemma \ref{fix0} it is enough to prove that the following condition holds.
\begin{quote}
For each $\delta>0$ there exists $\eta>0$ such that if $\mathbf{p}, \mathbf{q}\in H_i$, $|\varphi(\mathbf{p})-\mathbf{p}|<\eta$ and $|\varphi(\mathbf{q})-\mathbf{q}|<\eta$, then $|\mathbf{p}-\mathbf{q}|\leq \delta$.
\end{quote}
Assume that the set
\[K=\{(\mathbf{p},\mathbf{q})\in H_i\times H_i:\ |\mathbf{p}-\mathbf{q}|\geq \delta\}\]
is nonempty and compact\footnote{See Theorem 2.2.13 of \cite{bv}.}. Since the mapping $(\mathbf{p},\mathbf{q})\longrightarrow \max(\varphi(\mathbf{p})-\mathbf{p}|,|\varphi(\mathbf{q})-\mathbf{q}|)$ is uniformly continuous, we can construct an increasing binary sequence $(\lambda_n)_{n\geq 1}$ such that
\[\lambda_n=0\Rightarrow \inf_{(\mathbf{p},\mathbf{q})\in K}\max(|\varphi(\mathbf{p})-\mathbf{p}|,|\varphi(\mathbf{q})-\mathbf{q}|)<2^{-n},\]
\[\lambda_n=1\Rightarrow \inf_{(\mathbf{p},\mathbf{q})\in K}\max(|\varphi(\mathbf{p})-\mathbf{p}|,|\varphi(\mathbf{q})-\mathbf{q}|)>2^{-n-1}.\]
It suffices to find $n$ such that $\lambda_n=1$. In that case, if $|\varphi(\mathbf{p})-\mathbf{p}|<2^{-n-1}$, $|\varphi(\mathbf{q})-\mathbf{q}|<2^{-n-1}$, we have $(\mathbf{p},\mathbf{q})\notin K$ and $|\mathbf{p}-\mathbf{q}|\leq \delta$. Assume $\lambda_1=0$. If $\lambda_n=0$, choose $(\mathbf{p}_n, \mathbf{q}_n)\in K$ such that $\max(|\varphi(\mathbf{p}_n)-\mathbf{p}_n|, |\varphi(\mathbf{q}_n)-\mathbf{q}_n|)<2^{-n}$, and if $\lambda_n=1$, set $\mathbf{p}_n=\mathbf{q}_n=\mathbf{r}_n$. Then, $|\varphi(\mathbf{p}_n)-\mathbf{p}_n|\longrightarrow 0$ and $|\varphi(\mathbf{q}_n)-\mathbf{q}_n|\longrightarrow 0$, so $|\mathbf{p}_n-\mathbf{q}_n|\longrightarrow 0$. Computing $N$ such that $|\mathbf{p}_N-\mathbf{q}_N|<\delta$, we must have $\lambda_N=1$. Thus, we have completed the proof of the existence of a point which satisfies
\begin{equation}
\mathbf{p}=\varphi(\mathbf{p}).\label{con1}
\end{equation}

\item Let $\mathbf{p}^*=(p^*_0, p^*_1, \dots, p^*_n)$ be one of the points which satisfy (\ref{con1}). Let us consider the relationship between the price and excess demand of each good in that case. From the definitions of $\varphi$ and $\mathbf{v}$, and $\sum_{j=0}^nv_j=1+\sum_{j=0}^n\max(f_j,0)$(because $\sum_{j=0}^np^*_j=1$), (\ref{con1}) means
\[ \frac{p^*_i+\max(f_i,0)}{1+\sum_{j=0}^n\max(f_j,0)}=p^*_i\]
Let $\lambda=\sum_{j=0}^n\max(f_j,0)$. Then, we have
\[\max(f_i,0)=\lambda p^*_i.\]
From $\sum_{j=0}^np^*_j=1$ there is a $k$ such that $p^*_k>0$. If for all such $k$ $\max(f_k,0)=f_k=\lambda p^*_k>0$ holds, that is, excess demands for all goods with positive prices are positive, we can not cancel out $p^*_kf_k>0$ because the price of any good can not be negative, and the Walras law (\ref{wal}) is violated. Thus, we have $\lambda=0$ and
\begin{equation}
\max(f_i,0)=0\ \mathrm{for\ each}\ i. \label{eq1}
\end{equation}
This means that excess demand for each good is not positive. Such a state is an \emph{equilibrium} of a competitive economy. In the equilibrium when ${p}_i>0$ we must have $\varphi_i=0$ because if $\varphi_i<0$ we have ${p}_if_i<0$, and then the Walras law is violated. We have completed the proof of the existence of an equilibrium in a competitive economy with sequentially locally non-constant excess demand functions.

\end{enumerate}
\end{proof}

\section{From the existence of a competitive equilibrium to Sperner's lemma}

\begin{figure}[t]
\begin{center}
\includegraphics[height=7.5cm]{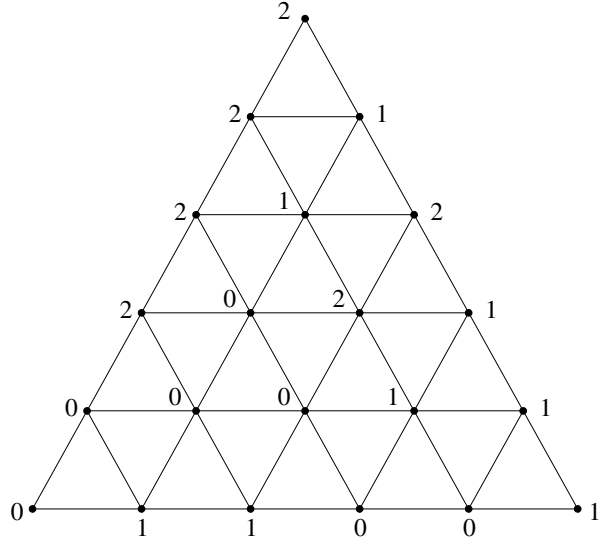}
\end{center}
	\vspace*{-.3cm}
	\caption{Partition and labeling of 2-dimensional simplex}
	\label{tria2}
\end{figure}

In this section we will derive Sperner's lemma from the existence of an equilibrium in a competitive economy\footnote{Our result in this section is a variant of Uzawa equivalence theorem (\cite{uzawa}) which (classically) states that the existence of a competitive equilibrium and Brouwer's fixed point theorem are equivalent.}. Let partition an $n$-dimensional simplex $\Delta$ in the way depicted in Figure \ref{tria2} for a 2-dimensional simplex. In a 2-dimensional case we divide each side of $\Delta$ in $m$ equal segments, and draw the lines parallel to the sides of $\Delta$. Then, the 2-dimensional simplex is partitioned into $m^2$ triangles. We consider partition of $\Delta$ inductively for cases of higher dimension. In a 3 dimensional case each face of $\Delta$ is a 2-dimensional simplex, and so it is partitioned into $m^2$ triangles in the above mentioned way, and draw the planes parallel to the faces of $\Delta$. Then, the 3-dimensional simplex is partitioned into $m^3$ trigonal pyramids. And similarly for cases of higher dimension. Denote the set of small $n$-dimensional simplices of $\Delta$ constructed by partition by $K$. Vertices of these small simplices of $K$ are labeled with the numbers 0, 1, 2, $\dots$, $n$ subject to the following rules. 
\begin{enumerate}
\item The vertices of $\Delta$ are respectively labeled with 0 to $n$. We label a point $(1,0, \dots, 0)$ with 0, a point $(0,1,0, \dots, 0)$ with 1, a point $(0,0,1 \dots, 0)$ with 2, $\dots$, a point $(0,\dots, 0,1)$ with $n$. That is, a vertex whose $k$-th coordinate ($k=0, 1, \dots, n$) is $1$ and all other coordinates are 0 is labeled with $k$ for all $k\in \{0, 1, \dots, n\}$. 

\item If a vertex of a simplex of $K$ is contained in an $n-1$-dimensional face of $\Delta$, then this vertex is labeled with some number which is the same as the number of a vertex of that face.

\item If a vertex of a simplex of $K$ is contained in an $n-2$-dimensional face of $\Delta$, then this vertex is labeled with some number which is the same as the number of a vertex of that face. And similarly for cases of lower dimension.

\item A vertex contained inside of $\Delta$ is labeled with an arbitrary number among 0, 1, $\dots$, $n$.
\end{enumerate}
Denote the vertices of an $n$-dimensional simplex of $K$ by $\mathbf{p}^0, \mathbf{p}^1, \dots, \mathbf{p}^n$, the $j$-th component of $\mathbf{p}^i$ by $\mathbf{p}^i_j$, and the label of $\mathbf{p}^i$ by $l(\mathbf{p}^i)$. Let $\tau$ be a positive number which is smaller than $\mathbf{p}^i_{l(\mathbf{p}^i)}$ for all $i$, and define a function $\varphi(\mathbf{p}^i)$ as follows\footnote{We refer to \cite{yoseloff} about the definition of this function.};
\[\varphi(\mathbf{p}^i)=(\varphi_0(\mathbf{p}^i), \varphi_1(\mathbf{p}^i), \dots, \varphi_n(\mathbf{p}^i)),\]
and
\begin{equation}
\varphi_j(\mathbf{p}^i)=\left\{
\begin{array}{ll}
\mathbf{p}^i_j-\tau&\mathrm{for}\ j=l(\mathbf{p}^i),\\
\mathbf{p}^i_j+\frac{\tau}{n}&\mathrm{for}\ j\neq l(\mathbf{p}^i).\label{e0}
\end{array}
\right.
\end{equation}
$\varphi_j$ denotes the $j$-th component of $\varphi$. From the labeling rules $\mathbf{p}^i_{l(\mathbf{p}^i)}>0$ for all $\mathbf{p}^i$, and so $\tau>0$ is well defined. Since $\sum_{j=0}^n\varphi_j(\mathbf{p}^i)=\sum_{j=0}^n\mathbf{p}^i_j=1$, we have
\[\varphi(\mathbf{p}^i)\in \Delta.\]
We extend $\varphi$ to all points in the simplex by convex combinations of its values on the vertices of the simplex. Let $\mathbf{q}$ be a point in the $n$-dimensional simplex of $K$ whose vertices are $\mathbf{p}^0, \mathbf{p}^1, \dots, \mathbf{p}^n$. Then, $\mathbf{q}$ and $\varphi(\mathbf{q})$ are represented as follows;
\[\mathbf{q}=\sum_{i=0}^n\lambda_i\mathbf{p}^i,\ \mathrm{and}\ \varphi(\mathbf{q})=\sum_{i=0}^n\lambda_i\varphi(\mathbf{p}^i),\ \lambda_i\geq 0,\ \sum_{i=0}^n\lambda_i=1.\]
It is clear that $\varphi$ is uniformly continuous. We verify that $\varphi$ is sequentially locally non-constant.
\begin{enumerate}
\item Assume that a point $\mathbf{r}$ is contained in an $n-1$-dimensional small simplex $\delta^{n-1}$ constructed by partition of an $n-1$-dimensional face of $\Delta$ such that its $i$-th coordinate is $\mathbf{r}_i=0$. Denote the vertices of $\delta^{n-1}$ by $\mathbf{r}^j,\ j=0, 1, \dots, n-1$ and their $i$-th coordinate by $\mathbf{r}_i^j$. Then, we have
\[\varphi_i(\mathbf{r})=\sum_{j=0}^{n-1}\lambda_j\varphi_i(\mathbf{r}^j),\ \lambda_j\geq 0,\ \sum_{j=0}^n\lambda_j=1.\]
Since all vertices of $\delta^{n-1}$ are not labeled with $i$, (\ref{e0}) means $\varphi_i(\mathbf{r}^j)>\mathbf{r}_i^j$ for all $j=\{0, 1, \dots, n-1\}$. Then, there exists no sequence $(\mathbf{r}_m)_{m\geq 1}$ such that $|\varphi(\mathbf{r}_m)-\mathbf{r}_m|\longrightarrow 0$ in an $n-1$-dimensional face of $\Delta$. Similarly for a simplex in a face of $\Delta$ of lower dimension.
\item Let $\mathbf{r}$ be a point in an $n$-dimensional simplex $\delta^n$ of $K$. Assume that no vertex of $\delta^n$ is labeled with $i$. Then
\begin{equation}
\varphi_i(\mathbf{r})=\sum_{j=0}^n\lambda_j\varphi_i(\mathbf{p}^j)=\mathbf{r}_i+\left(1+\frac{1}{n}\right)\tau,\label{e8}
\end{equation}
and so $\mathbf{r}\neq \varphi(\mathbf{r})$. Then, there exists no sequence $(\mathbf{r}_m)_{m\geq 1}$ such that $|\varphi(\mathbf{r}_m)-\mathbf{r}_m|\longrightarrow 0$ in $\delta^n$.
\item Assume that $\mathbf{r}$ is contained in a fully labeled $n$-dimensional simplex $\delta^n$, and rename vertices of $\delta^n$ so that a vertex $\mathbf{p}^i$ is labeled with $i$ for each $i$. Then,
\begin{align*}
\varphi_i(\mathbf{r})=\sum_{j=0}^n\lambda_j\varphi_i(\mathbf{p}^j)=\sum_{j=0}^n\lambda_j\mathbf{p}_i^j+\sum_{j\neq i}\lambda_j\frac{\tau}{n}-\lambda_i\tau=\mathbf{r}_i+\left(\frac{1}{n}\sum_{j\neq i}\lambda_j-\lambda_i\right)\tau\ \mathrm{for\ each}\ i.
\end{align*}
Consider sequences $(\mathbf{r}_m)_{m\geq 1}$ and $(\mathbf{r}'_m)_{m\geq 1}$ such that $|\varphi(\mathbf{r}_m)-\mathbf{r}_m|\longrightarrow 0$ and $|\varphi(\mathbf{r}'_m)-\mathbf{r}'_m|\longrightarrow 0$. Let $\mathbf{r}_m=\sum_{i=0}^n\lambda(m)_i\mathbf{p}^i$ with $\lambda(m)_i\geq 0,\ \sum_{i=0}^n\lambda(m)_i=1$ and $\mathbf{r}'_m=\sum_{i=0}^n\lambda'(m)_i\mathbf{p}^i$ with $\lambda'(m)_i\geq 0,\ \sum_{i=0}^n\lambda'(m)_i=1$. Then, we have
\[\frac{1}{n}\sum_{j\neq i}\lambda(m)_j-\lambda(m)_i\longrightarrow 0,\ \mathrm{and}\ \frac{1}{n}\sum_{j\neq i}\lambda'(m)_j-\lambda'(m)_i\longrightarrow 0\ \mathrm{for\ all}\ i.\]
Therefore, we obtain
\[\lambda(m)_i\longrightarrow \frac{1}{n+1}, \mathrm{and}\ \lambda'(m)_i\longrightarrow \frac{1}{n+1}.\]
These mean
\[|\mathbf{r}_m-\mathbf{r}'_m|\longrightarrow 0.\]
\end{enumerate}
Thus, $\varphi$ is sequentially locally non-constant

Now, using $\varphi$, we construct an excess demand function $g=(g_0, g_1, \dots, g_n)$ as follows;
\begin{equation}
g_i(\mathbf{p})=\varphi_i(\mathbf{p})-\mathbf{p}_i\mu(\mathbf{p}),\ i=0, 1, \dots, n.\label{uz1}
\end{equation}
$\mathbf{p}\in \Delta$, and $\mu(\mathbf{p})$ is defined by
\[ \mu(\mathbf{p})=\frac{\sum^n_{i=0}\mathbf{p}_i\varphi_i(\mathbf{p})}{\sum^n_{i=0}\mathbf{p}_i^2}.\]
$g$ is uniformly continuous, and satisfies the Walras law as shown below. Multiplying $\mathbf{p}_i$ to (\ref{uz1}) for each $i$, and adding them from 0 to $n$ yields
\begin{eqnarray}
\sum^{n}_{i=0}\mathbf{p}_ig_i&=&\sum^{n}_{i=0}\mathbf{p}_i\varphi_i(\mathbf{p})-\mu(\mathbf{p})\sum^{n}_{i=0}\mathbf{p}_i^2=\sum^{n}_{i=0}\mathbf{p}_i\varphi_i(\mathbf{p})-\frac{\sum^n_{i=0}\mathbf{p}_i\varphi_i(\mathbf{p})}{\sum^n_{i=0}\mathbf{p}_i^2}\sum^{n}_{i=0}\mathbf{p}_i^2 \notag \\
&=&\sum^{n}_{i=0}\mathbf{p}_i\varphi_i(\mathbf{p})-\sum^{n}_{i=0}\mathbf{p}_i\varphi_i(\mathbf{p})=0.\label{wal2}
\end{eqnarray}
Therefore, $g_i(\mathbf{p})$'s satisfy the conditions for excess demand functions, and because of sequential local non-constancy of $\varphi$ they are sequentially locally non-constant expressed as follows;.
\begin{quote}
Assume $\max(g_i, 0)\longrightarrow 0$ for all $i$. For (\ref{wal2}) to be satisfied $|\varphi_i(\mathbf{p})-p_i\mu(\mathbf{p})|\longrightarrow 0$ for $i$ such that $p_i>0$, and so $\mu(\mathbf{p})\longrightarrow 1$. Thus, $|\varphi_i(\mathbf{p})-p_i\mu|\longrightarrow 0$ for $i$ such that $p_i>0$. For $i$ such that $p_i=0$, $\varphi_i\geq 0$ with $\max(g_i, 0)\longrightarrow 0$ implies $|\varphi_i(\mathbf{p})-p_i\mu(\mathbf{p})|\longrightarrow 0$. 
\end{quote}
Therefore, sequential local non-constancy of $\varphi$ implies sequential local non-constancy of $g$, and there exists an equilibrium. Let $\mathbf{p}^*=\{\mathbf{p}_0^*,\mathbf{p}_1^*,\dots,\mathbf{p}_n^*\}$ be the equilibrium price vector. Then,
\[g_i(\mathbf{p}^*)\leq 0\ \mathrm{for\ all}\ i,\]
and
\[g_i(\mathbf{p}^*)=0\ \mathrm{for}\ i\ \mathrm{such\ that}\ \mathbf{p}^*_i>0,\]
and so $\varphi_i(\mathbf{p}^*)=\mu(\mathbf{p}^*)\mathbf{p}_i^*$ for all such $i$. $\sum_{i=0}^n\varphi_i(\mathbf{p}^*)=\sum_{i=0}^n\mathbf{p}_i^*=1$ implies $\mu(\mathbf{p}^*)\leq 1$. On the other hand, $g_i(\mathbf{p}^*)\leq 0$(for all $i$) means $\mu(\mathbf{p}^*)\geq 1$. Thus, $\mu(\mathbf{p}^*)=1$, and we obtain
\begin{equation}
\varphi_i(\mathbf{p}^*)=\mathbf{p}_i^*\ \mathrm{for\ all}\ i.\label{e5}
\end{equation}
Let $\delta^*$ be a simplex of $K$ which contains $\mathbf{p}^*$, and $\mathbf{p}^0, \mathbf{p}^1, \dots, \mathbf{p}^n$ be the vertices of $\delta^*$. Then, $\mathbf{p}^*$ and $\varphi(\mathbf{p}^*)$ are represented as
\[\mathbf{p}^*=\sum_{i=0}^n\lambda_i\mathbf{p}^i\ \mathrm{and}\ \varphi(\mathbf{p}^*)=\sum_{i=0}^n\lambda_i\varphi(\mathbf{p}^i),\ \lambda_i\geq 0,\ \sum_{i=0}^n\lambda_i=1.\]
(\ref{e0}) implies that if only one $\mathbf{p}^k$ among $\mathbf{p}^0, \mathbf{p}^1, \dots, \mathbf{p}^n$ is labeled with $i$, we have
\[ \varphi_i(\mathbf{p}^*)-\mathbf{p}^*_i=\sum_{j=0}^n\lambda_j\mathbf{p}_i^j+\sum_{j=0, j\neq k}\lambda_j\frac{\tau}{n}-\lambda_k\tau-\mathbf{p}_i^*=\left(\frac{1}{n}\sum_{j=0, j\neq k}\lambda_j-\lambda_k\right)\tau.\]
$\mathbf{p}^j_i$ is the $i$-th component of $\mathbf{p}^j$.

Since $\tau>0$, $\varphi_i(\mathbf{p}^*)=\mathbf{p}^*_i$(for all $i$) is equivalent to
\[\frac{1}{n}\sum_{j=0, j\neq k}\lambda_j-\lambda_k=0.\]
(\ref{e5}) is satisfied with $\lambda_k=\frac{1}{n+1}$ for all $k$. On the other hand, if no $\mathbf{p}^j$ is labeled with $i$, we have
\[\varphi_i(\mathbf{p}^*)=\sum_{j=0}^n\lambda_j\mathbf{p}_i^j=\mathbf{p}_i^*+\left(1+\frac{1}{n}\right)\tau,\]
and then (\ref{e5}) can not be satisfied. Thus, for each $i$ one and only one $\mathbf{p}^j$ must be labeled with $i$. Therefore, $\Delta^*$ must be a fully labeled simplex. We have completed the proof of Sperner's lemma.

\bibliography{yasuhito}

\end{document}